\documentclass[letterpaper]{article}
\usepackage[left=3cm, right=3cm, top=3cm]{geometry}

\usepackage{amssymb,amsthm,amsmath}
\usepackage{mathtools}
\usepackage{bbm}
\allowdisplaybreaks

\usepackage{thmtools,thm-restate}
\usepackage{mathtools}

\declaretheorem[]{theorem}
\declaretheorem[]{definition}

\declaretheorem[]{corollary}

\declaretheorem[]{remark}

\newcommand\dist{\mathrel{\overset{\makebox[0pt]{\mbox{\normalfont\tiny\sffamily d}}}{=}}}

\begin{document}

\title{Exponential Convergence Rates for Stochastically Ordered Markov Processes with Random Initial Conditions} 
\author{Julia Gaudio, Saurabh Amin, and Patrick Jaillet}


\maketitle
\begin{abstract}
In this brief paper we find computable exponential convergence rates for a large class of stochastically ordered Markov processes. We extend the result of Lund, Meyn, and Tweedie (1996), who found exponential convergence rates for stochastically ordered Markov processes starting from a fixed initial state, by allowing for a random initial condition that is also stochastically ordered. Our bounds are formulated in terms of moment-generating functions of hitting times. To illustrate our result, we find an explicit exponential convergence rate for an M/M/1 queue beginning in equilibrium and then experiencing a change in its arrival or departure rates, a setting which has not been studied to our knowledge. 
\end{abstract}


\section{Introduction}
This paper is concerned with parametrized stochastically ordered Markov processes. Consider, for example, a stable M/M/1 queue with service rate $\mu$ and arrival rate $\lambda < \mu$. For a fixed $\mu$, let $\{X_t(\pi, \lambda)\}_{t \geq 0}$ be the queue-length process with service rate $\lambda$ and initial distribution $\pi$. Then $X_t(\pi, \lambda)$ is stochastically increasing in $\lambda$, for all $t \geq 0$. That is, 
$$\mathbb{P}\left(X_t(\pi, \lambda)\geq x \right) \leq \mathbb{P}\left(X_t(\pi, \lambda') \geq x \right)$$ 
for all $x \in \mathbb{Z}_+$ if $\lambda \leq \lambda'$. Similarly, $X_t(\pi, \mu)$ is stochastically decreasing in $\mu$ for fixed $\lambda$ and $\pi$. The focus of our paper is to analyze the convergence of a parametrized stochastically ordered Markov process to its stationary distribution, when its initial state is distributed according to a stationary distribution for another parameter choice. This will be stated more precisely below.

The Markov process is described by its transition kernel and its initial distribution. We assume that the initial distribution is the stationary distribution associated with setting the parameter equal to $r_0$, and we let $r$ be the parameter setting of the transition kernel. The parameter change happens once, at $t = 0$. In other words, if $r=r_0$, the process is always in equilibrium, and when $r \neq r_0$, the system starts in the equilibrium associated with $r_0$ and transitions over time to the one associated with $r$. The equilibrium distributions are denoted by $\pi(r_0)$ and $\pi(r)$. When $r \neq r_0$ we say the system is ``perturbed.'' These Markov processes will be denoted by $X_t\left(r_0, r\right)$. We sometimes refer to the collection $\{X_t\left(r_0, r \right)\}_{r_0,r}$ as a ``system.'' Note that there could be multiple parameters. For example, to study an M/M/1 queue starting in the stationary distribution associated with $(\lambda, \mu)$, operating under parameters $(\lambda', \mu')$, we would have $r_0 = (\lambda, \mu)$ and $r = (\lambda', \mu')$. As in \cite{Lund1996}, we consider Markov processes that take value in $[0, \infty)$. In this paper, we consider the total variation distance between a parametrized continuous time Markov process and its stationary distribution. Recall the definition of total variation distance:

\begin{definition}
The total variation distance between two measure $P$ and $Q$ on state space $\Omega$ is given by $$\left \Vert P - Q \right \Vert_{\text{TV}} = \sup_{A \subset \Omega} \left |P(A) - Q(A) \right|.$$
\end{definition}
We seek a convergence bound of the form $$\left \Vert \mathcal{L} \left(X_t(r_0, r)\right)- \pi(r) \right \Vert_{\text{TV}} \leq C e^{-\alpha t}.$$ The value $\alpha$ is referred to as the ``convergence rate.''

Prior work in the area of the convergence of continuous-time Markov processes focuses on convergence assuming a particular deterministic initial state. However, this type of analysis is limiting, because the initial state of a process is often unknown. In situations where the initial state is unobservable, it may be more reasonable to assume a particular initial distribution rather than a particular initial state. Our extension of the result by \cite{Lund1996} allows one to analyze a system in equilibrium that undergoes a perturbation of its parameters, pushing it towards another equilibrium. For example, one might wish to analyze the effect of a disruption on a queue of customers waiting for service. The bounds in this paper would allow one to study how quickly the queue length process reaches the new equilibrium after being perturbed.

We start by reviewing the existing literature on the convergence of stochastically ordered Markov processes, focusing on a paper by Lund, Meyn and Tweedie (\cite{Lund1996}). We extend the result of \cite{Lund1996}, allowing the initial state of the system to be distributed according to a stationary distribution from the family of distributions parametrized by the system parameters. To illustrate the value of our result, we apply it to the analysis of perturbed M/M/1 queues. More importantly, our result applies to a broader class of Markov processes, namely any parametrized Markov process whose initial distribution is a stationary distribution.

\section{Related work}
Lund, Meyn, and Tweedie (\cite{Lund1996}) establish convergence rates for nonnegative Markov processes that are stochastically ordered in their initial state, starting from a fixed initial state. Examples of such Markov processes include: M/G/1 queues, birth-and-death processes, storage processes, insurance risk processes, and reflected diffusions. We reproduce here the main theorem, Theorem 2.1 from \cite{Lund1996}, which will be extended in this paper. 

\begin{theorem}(\cite{Lund1996})\label{LundThm2.1}
Suppose that $\{X_t \}$ is a Markov process on $\Omega = [0, \infty)$ that is stochastically increasing in its initial state, with parameter setting $r$. Let $\tau_0(x)$ be the hitting time to zero of $X_t$ given that $X_0 = x$, and let $\tau_0(\pi(r))$ be the hitting time to zero of $X_t$ given that $X_0$ is distributed according to the stationary distribution $\pi(r)$.

Let $\mathcal{L} \left(X_t(x,r)\right)$ be the law of $X_t$ given that $X_0 = x$. If $\mathbb{E} \left[ e^{\alpha \tau_0(x)} \right] < \infty$ for some $\alpha >0$ and some $x > 0$, then 
\begin{align}
\left \Vert \mathcal{L} \left(X_t(x, r) \right) - \pi(r) \right \Vert_{\text{TV}} &\leq \left(\mathbb{E} \left[ e^{\alpha \tau_0(x)} \right]  + \mathbb{E} \left[ e^{\alpha \tau_0(\pi(r))} \right] \right) e^{-\alpha t} \label{LundThm}
\end{align}
for every $x \geq 0$ and $t \geq 0$.
\end{theorem}
The significance of this theorem is that it provides computable rates of convergence for a large class of Markov processes by relating the total variation distance from equilibrium to moment generating functions of hitting times to zero. We extend this result to the situation where $X_0$ is distributed according to the stationary distribution corresponding to a different parameter choice. The proof is analogous to the one given in \cite{Lund1996} and is based on a coupling approach.

The second major result in \cite{Lund1996} is to connect a drift condition to the convergence rate in \eqref{LundThm}, which is Theorem 2.2 (i) in \cite{Lund1996}, reproduced below:
\begin{theorem}(\cite{Lund1996})\label{LundThm2.2}
Suppose that $\{X_t\}$ is a Markov process that is stochastically increasing in its initial state. Let $\mathcal{A}$ be the extended generator of the process. If there exists a drift function $V : \Omega \rightarrow [1, \infty)$ and constants $c > 0$ and $b < \infty$ such that for all $x \in \Omega$
\begin{align}
\mathcal{A} V(x) &\leq -cV(x) + b \mathbbm{1}_{\{0\}}(x) \label{drift}
\end{align}
then $\mathbb{E} \left[ e^{c \tau_0(x)} \right] < \infty$ for all $x >0$, which implies that \eqref{LundThm} holds for $\alpha \leq c$.
\end{theorem}
\noindent We also connect Theorem \ref{LundThm2.2} to our extension of Theorem \ref{LundThm2.1}. 

Theorem \ref{LundThm2.1} is applied to several univariate systems in \cite{Lund1996}: finite capacity stores, dam processes, diffusion models, periodic queues, and M/M/1 queues. Additionally, one multivariate system is considered in \cite{Lund1996}: two M/M/1 queues in series. 

The paper by Lund et al (1996) has inspired numerous related papers, some of which we reference here. Several directly apply the main results; for example Novak and Watson (2009) used Theorem \ref{LundThm2.1} to derive the convergence rate of an M/D/1 queue, \nocite{Novak2009} while \cite{Tweedie1998} applied Theorem \ref{LundThm2.2} to establish the convergence of the northwest truncation (square submatrix that includes the top left entry) of transition probability matrices. In a more applied work, Kiessler (2008) used the result of \cite{Lund1996} to prove the convergence of an estimator for traffic intensity. \nocite{Kiessler2008}

Other works build on the derivation of more general convergence bounds. For example, Liu et al (2007) applied the main theorem in order to get bounds on the best uniform convergence rate for strongly ergodic Markov chains. \nocite{Liu2007} Drawing on the stochastic monotonicity and coupling approach of Lund et al (1996), Roberts and Rosenthal (2017) derived convergence rate bounds for symmetric Langevin diffusions, while Sarantsev (2016) followed the Lyapunov function approach to find convergence rates for jump diffusions on the half-line. \nocite{Roberts2017, Sarantsev2016} Hou et al (2005) also used a coupling method, focusing on establishing subgeometric convergence rates. \nocite{Hou2005} In related work to \cite{Hou2005}, Liu et al (2010) established subgeometric convergence rates via first hitting times and drift functions. \nocite{Liu2010} Douc et al (2004) were able to generalize convergence bounds to time-inhomogeneous chains using coupling and drift conditions. \nocite{Douc2004}

Roberts and Tweedie (2000) found convergence bounds for stochastically ordered Markov processes, allowing for no minimal reachable element. \nocite{Roberts2000} Their work includes a convergence bound for a Markov process that starts in a random initial condition, and therefore has a similar purpose as this paper. However, their bound is stated in terms of a drift condition, which may be challenging to verify because it requires finding a drift function. On the other hand, our bound relies on a direct calculation of moment generating functions of hitting times. Rosenthal (2002) also derives a convergence bound for an initial distribution for more general chains, using drift and minorization conditions, via a coupling approach. \nocite{Rosenthal2002} Baxendale (2005) derives convergence bounds for geometrically ergodic Markov processes with an alternate approach to \cite{Lund1996}, though also using a drift condition. \nocite{Baxendale2005} Scott and Tweedie (1996), and Douc et al (2005) consider convergence in $f$-norm: Scott and Tweedie (1996) find geometric and subgeometric convergence rates, while Douc et al (2007) find convergence rates for subgeometrically ergodic Harris-recurrent Markov chains, allowing for no minimal atom. \nocite{Scott1996, Douc2007} Additionally, the drift condition in \cite{Lund1996} has appeared numerous times in literature on controlled Markov chains and Markov Decision Processes (e.g. \cite{Prieto-Rumeau2012}, \cite{Guo2007}, \cite{Guo2010}).

\section{Main result}
We begin with some definitions that we utilize in this paper. We let $\{X_t(r_0, r) \}$ denote the process governed by $r$ with initial distribution corresponding to $r_0$. Similarly, we let $\{X_t(r) |X_0(r) = x \}$ denote the process governed by $r$ with initial state $x$.

\begin{definition}
A set $A$ is said to be \emph{increasing} if
$$\forall x \in A, y \geq x \implies y \in A.$$
\end{definition}

\begin{definition}\label{stochasticallyincreasing}
For a family of nonnegative Markov processes $\{X_t(r_0, r) \}$ with transition kernel parametrized by $r$ with starting stationary distribution parametrized by $r_0$, we say that $X_t$ is \emph{stochastically increasing in $r_0$} if for all $t \geq 0$ and all increasing sets $A \subset \Omega$, 
$$\mathbb{P} \left(X_t(r_0, r) \in A\right) \leq \mathbb{P} \left(X_t(r_0', r)  \in A\right)$$
whenever $r_0 \leq r_0'$.  
Note that for a univariate process, $A$ is of the form $\{y  \in \Omega : y \geq x \}$ for some $x$.
\end{definition}

\begin{definition}
Define $\tau_0(r_0, r)$ to be the hitting time to the zero state of $\{X_t(r_0, r) \}$. Similarly, define $\tau_0(x, r)$ to be the hitting time to the zero state of $\{X_t(r) | X_0(r) = x \}$ For a Markov process $\{ X_t(r_0, r)\}$, let 
$$G(r_0, r, \alpha) = \mathbb{E} \left[e^{\alpha \tau_0(r_0, r)} \right]$$
and similarly for a Markov process $\{X_t(r) |X_0(r) = x \}$, define $G(x, r, \alpha) = \mathbb{E} \left[e^{\alpha \tau_0(x, r)} \right].$
\end{definition}

We now extend Theorem \ref{LundThm2.1} to allow for a random initial condition. 

\begin{theorem}\label{extension2.1}
Consider a family of nonnegative Markov processes $\{X_t(r_0, r) \}$ that is stochastically increasing in $r$, where $r = r_0$ corresponds to the system being in equilibrium. 
Let $r_m = \max\{r_0, r\}$. If $G(r_m, r, \alpha) < \infty$ for some $\alpha > 0$, then
\begin{align}
\left \Vert \mathcal{L}\left(X_t(r_0, r) \right) - \pi(r) \right \Vert_{\text{TV}} &\leq G(r_m, r, \alpha)  e^{-\alpha t} \label{extension2.1eq}
\end{align}
\end{theorem}

\begin{proof}
Note that $X_t(r,r) \dist \pi(r)$. Using the coupling inequality, we have 
\begin{align*}
\left \Vert \mathcal{L}\left(X_t(r_0, r) \right) - \pi(r) \right \Vert_{\text{TV}} &\leq \mathbb{P}\left(X_t(r_0, r) \neq X_t(r,r) \right)
\end{align*}
where $\left(X_t(r_0, r),X_t(r,r)  \right)$ is any coupling.\\
Either $\{X_t(r_m, r) \} \dist \{X_t(r_0, r) \}$ or $\{X_t(r_m, r) \} \dist \{X_t(r, r) \}$. We can create copies $X_t(r_0, r)'$, $X_t(r, r)'$ so that  $X_t(r_m, r) = X_t(r_0, r)' \geq X_t(r, r)'$ if $r_m = r_0$, and  $X_t(r_m, r) = X_t(r, r)' \geq X_t(r_0, r)'$ if $r_m = r$. This is possible by an extension of Strassen's Theorem to stochastic processes, developed in \cite{Kamae1977} and as cited by \cite{Lund1996}. We take $\left(X_t(r_0, r)',X_t(r,r)'  \right)$ as the coupling. Then, the process $X_t(r_m, r)$ acts as a bounding process. Observe that 
$$\left\{X_t(r_m, r) = 0\right\} \implies \left\{X_t(r_0, r) = X_t(r, r) = 0 \right\}$$
and the coupling occurs at or before time $t$. So then we have
$$ \mathbb{P}\left(X_t(r_0, r) \neq X_t(r,r) \right) \leq \mathbb{P} \left(\tau_0\left(r_m, r \right) > t \right).$$
 Exponentiating and using Markov's inequality, we obtain the desired result:
\begin{align*}
\left \Vert \mathcal{L}\left(X_t(r_0, r) \right) - \pi(r) \right \Vert_{\text{TV}} &\leq \mathbb{P} \left(e^{\alpha \tau_0(r_m, r)} > e^{\alpha t} \right) \text{ for $\alpha > 0$}\\
&\leq \mathbb{E} \left[e^{\alpha \tau_0(r_m, r)} \right] e^{-\alpha t}\\
&= G(r_m, r, \alpha) e^{-\alpha t}
\end{align*}
\end{proof}
However, the challenge in applying Theorem \ref{extension2.1} is finding $\alpha >0$ for which $G(r_m, r, \alpha)$ is finite. Note that $G(r_m, r, \alpha) $ is a moment generating function (MGF), so $\{\alpha : G(r_m, r, \alpha) < \infty\}$ is an interval containing zero, typically referred to as the \emph{domain}. For some Markov processes, the domain is precisely known. One such example is the M/M/1 queue with fixed service rate, where the arrival rate is perturbed from $r_0 = \lambda_0$ to $r = \lambda$. For processes where the domain is difficult to find but $r_m = r$, we can apply Theorem \ref{LundThm2.2}.

\begin{corollary}\label{cordrift}
If $r_m = r$ and the drift condition (\ref{drift}) holds for a Markov process $X_t(x, r)$ with some $V(\cdot), b, c$, then (\ref{extension2.1eq}) holds with $\alpha = c$.
\end{corollary}

\begin{proof}
If the drift condition holds then $G(x, r, \alpha) < \infty$, by Theorem \ref{LundThm2.2}. Applying Lemma 3.1 from \cite{Lund1996}, we also have that $G(r, r, c) < \infty$.
\end{proof}

We now apply Theorem \ref{extension2.1} to the analysis of a single M/M/1 queue.
\subsection{M/M/1 queues}
\subsubsection{Queue length process.}
We study the queue length process and consider perturbing the arrival and service rates from $(\lambda_0, \mu_0)$ to $(\lambda, \mu)$. Throughout, we assume the stability conditions $\lambda_0 < \mu_0$ and $\lambda < \mu$. First we consider the case of changing the arrival rate while keeping the service rate fixed. We then show how to find bounds for any change of the two parameters, as long as the service rate is greater than the arrival rate.

Supposing that $\mu = \mu_0$, we can suppress the service rate and write the two processes as $X_t(\lambda_0, \lambda)$ and $X_t(\lambda, \lambda)$.
Let $\lambda_m = \max \{\lambda_0, \lambda \}$. From Theorem \ref{extension2.1}, we have
\begin{align}
\left \Vert \mathcal{L}\left(X_t(\lambda_0, \lambda)\right) - \pi(\lambda) \right \Vert_{\text{TV}} &\leq G(\lambda_m, \lambda, \alpha) e^{-\alpha t} \label{applythm}
\end{align}
Let us analytically compute $G(\lambda_m, \lambda, \alpha)$. Let $\tau_y(x, \lambda)$ be the hitting time to $y$ of the M/M/1 queue with parameters set to $(\lambda, \mu)$, started from a queue length of $x$, and write $$G(x, \lambda, \alpha) = \mathbb{E} \left[e^{\alpha \tau_0(x)} \right].$$ 
Then by conditioning on the initial state, we obtain
\begin{align*}
G(\lambda_m, \lambda, \alpha) &= \mathbb{E} \left[ e^{\alpha \tau_0(\lambda_m, \lambda)} \right]\\
&= \sum_{x = 0}^{\infty} \left(1 - \frac{\lambda_m}{\mu} \right) \left( \frac{\lambda_m}{\mu} \right)^x G(x, \lambda, \alpha)
\end{align*}
Now by decomposing the hitting time and noting the independence and stationarity of the incremental hitting times,
\begin{align*}
G(x, \lambda, \alpha) &= \mathbb{E} \left[e^{\alpha \tau_0(x, \lambda)} \right]\\
&=  \mathbb{E} \left[\prod_{i = 1}^x e^{\alpha  \tau_{x-i}(x - i + 1, \lambda)} \right]\\
& = \prod_{i = 1}^x \mathbb{E} \left[e^{\alpha  \tau_{x-i}(x - i + 1, \lambda)} \right]\\
& = \prod_{i = 1}^x \mathbb{E} \left[e^{\alpha  \tau_{0}(1, \lambda)} \right]\\
&= \left(G(1, \lambda, \alpha)\right)^x
\end{align*}

Therefore 
\begin{align}
G(\lambda_m, \lambda, \alpha) &= \sum_{x = 0}^{\infty} \left(1 - \frac{\lambda_m}{\mu} \right) \left( \frac{\lambda_m}{\mu} \right)^x \left(G(1, \lambda, \alpha)\right)^x\nonumber\\
&=  \frac{1 - \frac{\lambda_m}{\mu}}{1 - \frac{\lambda_m}{\mu} G(1, \lambda, \alpha)} \label{Geq}
\end{align}
as long as $\frac{\lambda_m}{\mu} G(1, \lambda, \alpha) < 1$. Next we compute $G(1, \lambda, \alpha)$.

\begin{theorem}\label{MGF}
Assume $\lambda < \mu$. For $\alpha \leq \left(\sqrt{\mu} - \sqrt{\lambda} \right)^2$,
\begin{align}
G(1, \lambda, \alpha) &= \mathbb{E} \left[e^{\alpha  \tau_{0}(1, \lambda)} \right] =\frac{1}{2 \lambda} \left(\lambda + \mu - \alpha - \sqrt{\left(\lambda + \mu - \alpha \right)^2 - 4 \lambda \mu} \right) \label{eqMGF}
\end{align}
\end{theorem}

\begin{proof}
In order to calculate the MGF, we condition on whether a departure or an arrival happens first. Let $E_A$ be the event that an arrival happens first and let $E_D$ be the event that a departure happens first. Let $\tau(A, \lambda)$ be the time required for the arrival, conditioned on the an arrival happening first; we define $\tau(D, \lambda)$ similarly. Using properties of exponential random variables, we have
\begin{align*}
\mathbb{E} \left[e^{\alpha  \tau_{0}(1, \lambda)} \right]&= \mathbb{E} \left[e^{\alpha \tau_0(1, \lambda)} | E_A\right] \mathbb{P}(E_A) + \mathbb{E} \left[e^{\alpha \tau_0(1, \lambda)} | E_D\right] \mathbb{P}(E_D)\\
&= \mathbb{E} \left[e^{\alpha (\tau_0(2, \lambda) + \tau(A, \lambda))} \right] \frac{\lambda}{\lambda + \mu} + \mathbb{E} \left[e^{\alpha \tau(D, \lambda)}\right] \frac{\mu}{\lambda + \mu}\\
&= \mathbb{E} \left[e^{\alpha \tau_0(1, \lambda)}\right]^2  \mathbb{E} \left[e^{\alpha \tau(A, \lambda)}\right] \frac{\lambda}{\lambda + \mu} + \mathbb{E} \left[e^{\alpha \tau(D, \lambda)}\right] \frac{\mu}{\lambda + \mu}
\end{align*}
Now since $\tau(A, \lambda) \dist \tau(D, \lambda) \sim exp(\lambda + \mu)$, 
$$\mathbb{E} \left[e^{\alpha \tau(A, \lambda)}\right] = \mathbb{E} \left[e^{\alpha \tau(D, \lambda)}\right] = \frac{\lambda + \mu}{\lambda + \mu - \alpha}$$ so long as $\alpha < \lambda + \mu$. In fact, this is the case: $\alpha \leq \left( \sqrt{\mu} - \sqrt{\lambda} \right)^2 = \mu + \lambda - 2 \sqrt{\lambda \mu} < \lambda + \mu$.
Now in order to find $\mathbb{E} \left[e^{\alpha \tau_0(1, \lambda)} \right] $ we solve the resulting quadratic to obtain
\begin{align}
\mathbb{E} \left[e^{\alpha \tau_0(1, \lambda)} \right] = \frac{1}{2\lambda} \left(\lambda + \mu - \alpha \pm \sqrt{(\lambda + \mu - \alpha)^2 - 4 \lambda \mu} \right) \label{G1}
\end{align}
In order for the discriminant to be nonnegative, we need $\alpha \leq \left(\sqrt{\mu} - \sqrt{\lambda} \right)^2$ or $\alpha \geq \left(\sqrt{\mu} + \sqrt{\lambda}\right)^2$. However, the second condition is overruled by the condition $\alpha < \lambda + \mu$.
To identify the correct root, we use the differentiation property of moment generating functions:
\begin{align*}
\mathbb{E} \left[ \tau_0(1, \lambda) \right] = \frac{d}{d \alpha} \mathbb{E}\left[ e^{\alpha \tau_0(1, \lambda)} \right] \bigg\rvert_{\alpha = 0}
\end{align*}
Again conditioning on whether an arrival or departure happens first, we have
\begin{align*}
\mathbb{E} \left[ \tau_0(1, \lambda) \right] &= \left(\mathbb{E} \left[ \tau_0(2, \lambda) \right] + \frac{1}{\lambda + \mu} \right) \frac{\lambda}{\lambda + \mu} + \left( \frac{1}{\lambda + \mu} \right) \frac{\mu}{\lambda + \mu}\\
&=\mathbb{E} \left[ \tau_0(1, \lambda) \right] = \left(2\mathbb{E} \left[ \tau_0(1, \lambda) \right] + \frac{1}{\lambda + \mu} \right) \frac{\lambda}{\lambda + \mu} + \left( \frac{1}{\lambda + \mu} \right) \frac{\mu}{\lambda + \mu}\\
\implies \mathbb{E} \left[ \tau_0(1, \lambda) \right] &= \frac{1}{\mu - \lambda}
\end{align*}
The $+$ root of Equation \eqref{G1} gives 
$\frac{d}{d \alpha} \mathbb{E}\left[ e^{\alpha \tau_0(1, \lambda)} \right] \bigg\rvert_{\alpha = 0} = \frac{\mu}{\lambda (\lambda - \mu)} < 0$
and the $-$ root gives
$\frac{d}{d \alpha} \mathbb{E}\left[ e^{\alpha \tau_0(1, \lambda)} \right] \bigg\rvert_{\alpha = 0} = \frac{1}{\mu - \lambda} =  \mathbb{E} \left[ \tau_0(1, \lambda) \right].$
This concludes the proof.
\end{proof}
\begin{remark}
After proving Theorem \ref{MGF}, we came to know of an alternate proof in \cite{Prabhu1998}, pp. 92-95.
\end{remark}
Now we apply Theorem \ref{extension2.1} to the convergence of M/M/1 queue with arrival rate perturbed from $r_0 = \lambda_0$ to $r = \lambda$, using Theorem \ref{MGF}. There are two cases:\\
\noindent \textbf{Case 1:} $\lambda_m = \lambda \geq \lambda_0$\\
Set  $\alpha = \left(\sqrt{\mu} - \sqrt{\lambda} \right)^2$ in Equation \eqref{eqMGF} to obtain
\begin{align*}
G(1, \lambda, \alpha) = \sqrt{\frac{\mu}{\lambda}}.
\end{align*}
To substitute into Equation \eqref{Geq}, we need to verify that $\frac{\lambda_m}{\mu} G(1, \lambda, \alpha) < 1$.
$$\frac{\lambda_m}{\mu} G(1, \lambda, \alpha) = \frac{\lambda}{\mu} \sqrt{\frac{\mu}{\lambda}}
=  \sqrt{\frac{\lambda}{\mu}}<1.$$
Thus, we obtain 
$$G(\lambda_m, \lambda, \alpha)=  \frac{1 - \frac{\lambda}{\mu} }{1 - \sqrt{\frac{\lambda}{\mu}}}= 1 + \sqrt{\frac{\lambda}{\mu}}$$
and
\begin{align*}
\left \Vert \mathcal{L}\left(X_t(\lambda_0, \lambda)\right) - \pi(\lambda) \right \Vert_{\text{TV}}&\leq \left(1 + \sqrt{\frac{\lambda}{\mu}} \right) e^{-\left(\sqrt{\mu} - \sqrt{\lambda} \right)^2 t}.
\end{align*}

\noindent \textbf{Case 2:} $\lambda_m = \lambda_0 > \lambda$\\
We need to pick $\alpha$ for which 1) $\frac{\lambda_0}{\mu} G(1, \lambda, \alpha)< 1$ and 2) $\alpha \leq \left(\sqrt{\mu} - \sqrt{\lambda} \right)^2$.
Condition 1) is equivalent to
\begin{align*}
&\frac{\lambda_0}{\mu} \left(\frac{1}{2 \lambda} \left(\lambda + \mu - \alpha - \sqrt{\left(\lambda + \mu - \alpha \right)^2 - 4 \lambda \mu} \right) \right) < 1\\
&\sqrt{\left(\lambda + \mu - \alpha \right)^2 - 4 \lambda \mu}  >- \frac{2 \lambda \mu}{\lambda_0} + \lambda + \mu - \alpha
\end{align*}
To determine when Condition 1) holds, we set these quantities equal to each other. 
\begin{align*}
&\sqrt{\left(\lambda + \mu - \alpha \right)^2 - 4 \lambda \mu}  =- \frac{2 \lambda \mu}{\lambda_0} + \lambda + \mu - \alpha\\
&\left(\lambda + \mu - \alpha \right)^2 - 4 \lambda \mu  = \left(- \frac{2 \lambda \mu}{\lambda_0} + \lambda + \mu - \alpha \right)^2\\
&\alpha = \lambda + \mu - \lambda_0 - \frac{\lambda \mu}{\lambda_0}
\end{align*}
Squaring may have introduced additional solutions. With this value of $\alpha$, the left side is equal to
\begin{align*}
\sqrt{\left(\lambda + \mu - \alpha \right)^2 - 4 \lambda \mu}  &= \sqrt{\left(\lambda_0 + \frac{\lambda \mu}{\lambda_0} \right)^2 - 4 \lambda \mu}  \\
&= \sqrt{\left(\lambda_0 - \frac{\lambda \mu}{\lambda_0} \right)^2}\\
&= \left|\lambda_0 - \frac{\lambda \mu}{\lambda_0} \right|
\end{align*}
The right side is equal to $\lambda_0 - \frac{\lambda \mu}{\lambda_0}$. If $\lambda_0 > \sqrt{\lambda \mu}$ there is a solution, otherwise there is no solution. Setting $\alpha = 0$, the left side is equal to $\mu - \lambda$, while the right side is less than $\mu - \lambda$ (setting $\lambda_0 = \mu - \epsilon$). Therefore when $\lambda_0 > \sqrt{\lambda \mu}$, we pick $\alpha < \lambda + \mu - \lambda_0 - \frac{\lambda \mu}{\lambda_0}$. We verify that $\lambda + \mu - \lambda_0 - \frac{\lambda \mu}{\lambda_0} \leq \left(\sqrt{\mu} - \sqrt{\lambda} \right)^2$.  Otherwise, when $\lambda_0 \leq \sqrt{\lambda \mu}$, we are free to pick $\alpha = \left(\sqrt{\mu} - \sqrt{\lambda} \right)^2$. 

%

Therefore Theorem \ref{extension2.1} is satisfied by substituting either $\alpha = \lambda + \mu - \lambda_0 - \frac{\lambda \mu}{\lambda_0} - \epsilon$ or $\alpha = \left( \sqrt{\mu} - \sqrt{\lambda} \right)^2$, depending on the value of $\lambda_0$. Intuitively, large values of $\lambda_0$ correspond to more ``contraction'' when the system goes to equilibrium, and therefore the convergence rate $\alpha$ should be smaller. 


\begin{remark}
The function
$$f(\lambda_0) = \begin{cases}
\left(\sqrt{\mu} - \sqrt{\lambda} \right)^2 &\text { if } \lambda_0 \leq \sqrt{\lambda \mu}\\
\lambda + \mu - \lambda_0 - \frac{\lambda \mu}{\lambda_0}  &\text{ if } \lambda_0 > \sqrt{\lambda \mu}
\end{cases} $$
is continuous in $\lambda_0$. In other words, the convergence rate changes continuously in $\lambda_0$.
\end{remark}

\begin{remark}
The rate $\alpha^{\star} = \left( \sqrt{\mu} - \sqrt{\lambda} \right)^2$ is well-known as the best convergence rate for the M/M/1 queue length process starting in a fixed initial condition (see e.g. \cite{VanDoorn1985} in addition to \cite{Lund1996}). However, it is not immediately clear that the same result would hold in our setting where the initial state of the queue has a distribution, 
$$\left \Vert \mathcal{L} \left(X_t(\lambda_0, \lambda) \right) - \pi(\lambda) \right \Vert_{\text{TV}} \ngeq \mathbb{E}_{X \sim \lambda} \left[\left \Vert \mathcal{L} \left(X_t(\lambda) | X_t = X \right) - \pi(\lambda) \right \Vert_{\text{TV}} \right].$$
In other words, we cannot simply go from quenched to annealed convergence.
\end{remark}

In the Appendix, we show another technique that gives a convergence rate of $$\overline{\alpha} = \frac{\log{\frac{\mu}{\lambda_0}}}{\log \sqrt{\frac{\mu}{\lambda}}} \left( \sqrt{\mu} - \sqrt{\lambda} \right)^2$$ when $\lambda_0 > \sqrt{\lambda \mu}$. Therefore, the best known convergence rate in the $\lambda_0 > \sqrt{\lambda \mu}$ case is 
$$\max \left \{ \lambda + \mu - \lambda_0 - \frac{\lambda \mu}{\lambda_0}, \frac{\log{\frac{\mu}{\lambda_0}}}{\log \sqrt{\frac{\mu}{\lambda}}} \left( \sqrt{\mu} - \sqrt{\lambda} \right)^2\right \}.$$

\vspace{12pt}
We now consider a more general perturbation. Suppose the parameters of the M/M/1 queue change from $(\lambda_0, \mu_0)$ to $(\lambda, \mu)$. We can relate these parameters by 
$ab \lambda_0 = \lambda$ and $b \mu_0= \mu$.
Interpreting multiplication by $b$ as rescaling time by a factor of $b$, we can write
\begin{align*}
&\left \Vert \mathcal{L} \left(X_t((\mu_0, \lambda_0), (b \mu_0, ab \lambda_0) \right) - \pi\left( (b \mu_0, ab \lambda_0) \right) \right \Vert_{\text{TV}} \\
&=\left \Vert \mathcal{L} \left(X_{bt}((\mu_0, \lambda_0), (\mu_0, a \lambda_0) \right) - \pi\left( (\mu_0, a \lambda_0) \right) \right \Vert_{\text{TV}}.
\end{align*}
We then conclude
\begin{align*}
&\left \Vert \mathcal{L} \left(X_t((\mu_0, \lambda_0), (\mu, \lambda) \right) - \pi\left( (\mu, \lambda) \right) \right \Vert_{\text{TV}}\leq G\left( (\mu_0, \lambda_0), (\mu_0, \lambda_m), \alpha \right) e^{-\alpha bt}
\end{align*}
where $\lambda_m = \max \{\lambda_0, a \lambda_0\}$. Thus, we are left with $G\left( (\mu_0, \lambda_0), (\mu_0, \lambda_m), \alpha \right)$ which is of the same form as Equation \eqref{applythm}, allowing us to apply Theorem \ref{extension2.1} and Theorem \ref{MGF} in order to calculate a bound.

\subsubsection{Workload process.}
Next we consider the workload process, $\{W_t \}$, for an M/M/1 queue. The value $W_t \in \mathbb{R}_{\geq 0}$ is the time remaining until the queue is empty, starting from time $t$. As for the queue length process, we consider changing the arrival rate from $\lambda_0$ to $\lambda$ while keeping the service rate fixed at $\mu_0$. The process $\{W_t \}$ is stochastically increasing in $\lambda$. Applying  Theorem \ref{extension2.1}, we need to calculate $G_W(\lambda_m, \lambda, \alpha)$ for the process $\{W_t \}$. But $\{W_t = 0 \} = \{X_t = 0\}$ since the workload is zero if and only if the queue length is zero. Therefore $G_W(\lambda_m, \lambda, \alpha)= G_X(\lambda_m, \lambda, \alpha)$, and the same convergence results follow.

In \cite{Lund1996}, it is shown that $\alpha^{\star} = \left( \sqrt{\mu} - \sqrt{\lambda} \right)^2$ is the best possible convergence rate for the M/M/1 workload process beginning with initial condition $W_0 = 0$. Precisely, \cite{Lund1996} show that if $\alpha > \alpha^{\star}$ and $W_0 = 0$, 
$$\limsup_{t \to \infty} e^{\alpha t} \left \Vert \mathcal{L}(W_t) - \pi \right \Vert_{\text{TV}} = \infty.$$

We investigate whether a similar property holds when $W_0$ is distributed according to the parameters $(\mu_0, \lambda_0)$. When $\lambda_0 \leq \sqrt{\lambda \mu}$, it turns out that $\alpha^{\star} = \left( \sqrt{\mu} - \sqrt{\lambda} \right)^2$ is in fact the best rate. We use the bounding process idea again with $W_t(\lambda_0, \lambda)$ and $W_t(\lambda, \lambda)$,  which is analogous to the proof of Theorem 2.3 in \cite{Lund1996}. Let $T = \inf_t \{t : W_t(\lambda_0, \lambda) = W_t(\lambda, \lambda)\}$.

\begin{align*}
 &\left \Vert \mathcal{L}(W_t(\lambda_0, \lambda)) - \pi(\lambda) \right \Vert_{\text{TV}} \\
 &= \sup_A \left| \mathbb{P} \left(W_t(\lambda_0, \lambda) \in A \right) - \pi(A ; \lambda) \right|\\
 &\geq \left| \mathbb{P} \left(W_t(\lambda_0, \lambda) = 0 \right) - \pi(0; \lambda) \right|\\
 &= \mathbb{P}\left(W_t(\min\{\lambda_0, \lambda\}, \lambda) = 0, T > t \right)\\
 &\geq \mathbb{P}\left(W_t(\min\{\lambda_0, \lambda\}, \lambda) = 0, T > t  | W_0(\min\{\lambda_0, \lambda\}, \lambda) = 0\right) \mathbb{P} \left(W(\min\{\lambda_0, \lambda\}, \lambda) = 0 \right)\\
 &= \mathbb{P}\left(W_t(\lambda) = 0, T > t  | W_0(\lambda) = 0\right) \left( 1 - \frac{\lambda_0}{\mu}\right)
 \end{align*}

It is shown in the proof of Theorem 2.3 in \cite{Lund1996} that for $\alpha > \left( \sqrt{\mu} - \sqrt{\lambda}\right)^2$, 
$$\limsup_{t \to \infty} e^{\alpha t} \mathbb{P}\left(W_t(\lambda) = 0, T > t  | W_0(\lambda) = 0\right) = \infty.$$
Multiplying the left side by the constant $\left( 1 - \frac{\lambda_0}{\mu}\right)$, 
$$\limsup_{t \to \infty} e^{\alpha t} \mathbb{P}\left(W_t(\lambda) = 0, T > t  | W_0(\lambda) = 0\right) \left( 1 - \frac{\lambda_0}{\mu}\right) = \infty,$$ and we conclude that 
$$\limsup_{t \to \infty} e^{\alpha t}  \left \Vert \mathcal{L}(W_t(\lambda_0, \lambda)) - \pi(\lambda) \right \Vert_{\text{TV}} = \infty $$ when $\alpha > \left( \sqrt{\mu} - \sqrt{\lambda} \right)^2$.

When $\lambda_0 \geq \sqrt{\lambda \mu}$, we have a gap between the best known rate 
$$\alpha = \max \left \{ \frac{\log{\frac{\mu}{\lambda_0}}}{\log \sqrt{\frac{\mu}{\lambda}}} \left( \sqrt{\mu} - \sqrt{\lambda} \right)^2  , \lambda + \mu - \lambda_0 - \frac{\lambda \mu}{\lambda_0} \right\}$$ and the upper bound on the rate, $\alpha^{\star} = \left( \sqrt{\mu} - \sqrt{\lambda} \right)^2$. 

\section{Conclusion}
In this paper we presented a method for finding exponential convergence rates for stochastically ordered Markov processes with a random initial condition. This method of analysis is useful for perturbation analysis of Markov processes, such as various queueing systems. Furthermore, we provide an explicit exponential bound for convergence in total variation distance of an M/M/1 queue that begins in an equilibrium distribution. The method developed in this paper can certainly be applied to other systems, such as M/G/1 queues, as long as one can identify the domain of the moment generating function of the hitting time to the zero state. 

\appendix
\section*{Appendix}
Using a truncation technique, we can improve the convergence of the M/M/1 queue-length process (and therefore the workload process as well) in the case $\lambda_0 > \sqrt{\lambda \mu}$.
\begin{theorem}
There exists a computable $C$ such that 
$$\left \Vert \mathcal{L}(X_t(\lambda_0, \lambda)) - \pi(\lambda) \right \Vert_{TV} \leq C e^{-\overline{\alpha} t}$$
where $$\overline{\alpha} = \frac{\log{\frac{\mu}{\lambda_0}}}{\log \sqrt{\frac{\mu}{\lambda}}} \left( \sqrt{\mu} - \sqrt{\lambda} \right)^2.$$
\end{theorem}

\begin{proof}
\begin{align*}
\left \Vert \mathcal{L}(X_t(\lambda_0, \lambda)) - \pi(\lambda) \right \Vert_{TV} &= \sup_A \left |\mathbb{P} \left( X_t(\lambda_0, \lambda) \in A \right) - \pi(A; \lambda) \right|\\
&= \sup_A \mathbb{P} \left( X_t(\lambda_0, \lambda) \in A \right) - \pi(A; \lambda) \\
&= \sup_A \sum_{x=0}^{\infty} \mathbb{P} \left( X_t(\lambda) \in A | X_0 = x \right)\pi(x; \lambda_0) - \pi(A; \lambda) \\
&= \sup_A \sum_{x=0}^{\infty} \pi(x; \lambda_0) \left[\mathbb{P} \left( X_t(\lambda) \in A | X_0 = x \right) - \pi(A; \lambda) \right]
\end{align*}
We now truncate $\pi(\lambda_0)$. Let $N(\epsilon) = \min \{ N: \sum_{x = N+1}^{\infty} \pi(x; \lambda_0) \leq \epsilon\}$.
Continuing, 
\begin{align}
&\leq \sup_A \sum_{x=0}^{N(\epsilon)}\left[ \pi(x; \lambda_0) \left(\mathbb{P} \left( X_t(\lambda) \in A | X_0 = x \right) - \pi(A; \lambda) \right)\right] + \epsilon \label{loose1}\\
&\leq \sum_{x=0}^{N(\epsilon)}\left[ \pi(x; \lambda_0) \sup_A  \left(\mathbb{P} \left( X_t(\lambda) \in A | X_0 = x \right) - \pi(A; \lambda) \right) \right]+ \epsilon \label{loose2}
\end{align}
Let $\alpha^{\star} = \left( \sqrt{\mu} - \sqrt{\lambda} \right)^2$. Applying Theorem 2.1 from \cite{Lund1996}, we can write
\begin{align}
&\leq \sum_{x=0}^{N(\epsilon)} \left[ \pi(x; \lambda_0) \left(G(x, \lambda, \alpha) + G(\lambda, \lambda, \alpha^{\star}) \right)e^{-\alpha^{\star}t} \right] + \epsilon \nonumber \\
&\leq (1- \epsilon) \left(1 + \sqrt{\frac{\lambda}{\mu}}\right) e^{-\alpha^{\star}t} + \epsilon + \sum_{x=0}^{N(\epsilon)} \left[ \pi(x; \lambda_0) G(x, \lambda, \alpha) e^{-\alpha^{\star}t} \right] \nonumber \\
&= (1- \epsilon) \left(1 + \sqrt{\frac{\lambda}{\mu}}\right) e^{-\alpha^{\star}t} + \epsilon + \sum_{x=0}^{N(\epsilon)} \left[ \pi(x; \lambda_0) \left(G(1, \lambda, \alpha) \right)^x e^{-\alpha^{\star}t} \right] \nonumber \\
&= (1- \epsilon) \left(1 + \sqrt{\frac{\lambda}{\mu}}\right) e^{-\alpha^{\star}t} + \epsilon + \sum_{x=0}^{N(\epsilon)} \left[ \left(1 - \frac{\lambda_0}{\mu}\right) \left(\frac{\lambda_0}{\mu} \right)^x \left(\sqrt{\frac{\mu}{\lambda}}\right)^x e^{-\alpha^{\star}t} \right] \nonumber \\
&= (1- \epsilon) \left(1 + \sqrt{\frac{\lambda}{\mu}}\right) e^{-\alpha^{\star}t} + \epsilon +  \left(1 - \frac{\lambda_0}{\mu}\right) e^{-\alpha^{\star}t}\sum_{x=0}^{N(\epsilon)}  \left(\frac{\lambda_0}{\sqrt{\lambda \mu}} \right)^x  \nonumber \\
&= (1- \epsilon) \left(1 + \sqrt{\frac{\lambda}{\mu}}\right) e^{-\alpha^{\star}t} + \epsilon +  \left(1 - \frac{\lambda_0}{\mu}\right)  \left(\frac{\left(\frac{\lambda_0}{\sqrt{\lambda \mu}} \right)^{N(\epsilon)+1} - 1}{\frac{\lambda_0}{\sqrt{\lambda \mu}}-1 } \right)e^{-\alpha^{\star}t}  \label{bound}
\end{align}
Set $\epsilon = e^{-\overline{\alpha}t}$ in order to fold in the $\epsilon$ term into a convergence bound. Then $N(\epsilon)$ must satisfy
\begin{align*}
\left(1-\frac{\lambda_0}{\mu}\right)\sum_{x=0}^{N(\epsilon)}  \left(\frac{\lambda_0}{\mu} \right)^x  &\geq 1-e^{-\overline{\alpha} t}\\
\left(1 - \frac{\lambda_0}{\mu} \right) \left( \frac{1- \left( \frac{\lambda_0}{\mu} \right)^{N(\epsilon)+1} }{1-\frac{\lambda_0}{\mu}} \right) &\geq 1 - e^{-\overline{\alpha}t}\\
\left( \frac{\lambda_0}{\mu} \right)^{N(\epsilon)+1} &\leq e^{-\overline{\alpha}t}\\
N(\epsilon) &\geq \frac{1}{\log \frac{\mu}{\lambda_0}} \overline{\alpha}t -1
\end{align*}
Substituting the value  $N(\epsilon) = \frac{1}{\log \frac{\mu}{\lambda_0}} \overline{\alpha}t \geq \left \lceil \frac{1}{\log \frac{\mu}{\lambda_0}} \overline{\alpha}t -1 \right \rceil$ back into the bound \eqref{bound}, the last term in the bound becomes
\begin{align*}
\left(1 - \frac{\lambda_0}{\mu}\right)  \left(\frac{\left(\frac{\lambda_0}{\sqrt{\lambda \mu}} \right)^{\frac{1}{\log \frac{\mu}{\lambda_0}} \overline{\alpha}t + 1} - 1}{\frac{\lambda_0}{\sqrt{\lambda \mu}}-1 } \right)e^{-\alpha^{\star}t} &= \left(1 - \frac{\lambda_0}{\mu}\right)  \left(\frac{ \frac{\lambda_0}{\sqrt{\lambda \mu}}e^{\log \left(\frac{\lambda_0}{\sqrt{\lambda \mu}} \right)\frac{1}{\log \frac{\mu}{\lambda_0}} \overline{\alpha}t} - 1}{\frac{\lambda_0}{\sqrt{\lambda \mu}}-1 } \right)e^{-\alpha^{\star}t} 
\end{align*}
If $\log \left(\frac{\lambda_0}{\sqrt{\lambda \mu}} \right)\frac{1}{\log \frac{\mu}{\lambda_0}} \overline{\alpha} < \alpha^{\star}$, then we get convergence at rate $$\min \left\{\overline{\alpha}, \alpha^{\star} - \frac{\log \left(\frac{\lambda_0}{\sqrt{\lambda \mu}} \right)}{\log \frac{\mu}{\lambda_0}} \overline{\alpha} \right \}$$.

Let $\overline{\alpha} = c \alpha^{\star}$ with $c <  \frac{\log \frac{\mu}{\lambda_0}} {\log \left( \frac{\lambda_0}{\sqrt{\lambda \mu}} \right)}$. Then we seek to maximize
$$\min \left \{c \alpha^{\star}, \alpha^{\star} - \frac{\log \left( \frac{\lambda_0}{\sqrt{\lambda \mu}} \right)}{\log \frac{\mu}{\lambda_0}} c \alpha^{\star}\right\}$$ over $c$. When $\lambda_0 > \sqrt{\lambda \mu}$ the factor $\frac{\log \left( \frac{\lambda_0}{\sqrt{\lambda \mu}} \right)}{\log \frac{\mu}{\lambda_0}}$ is positive, and the optimal $c$ is found by setting the two quantities equal to each other, leading to $c = \frac{\log{\frac{\mu}{\lambda_0}}}{\log \sqrt{\frac{\mu}{\lambda}}}$.
We verify that this value is less than $\frac{\log \frac{\mu}{\lambda_0}} {\log \left( \frac{\lambda_0}{\sqrt{\lambda \mu}} \right)}$.
Therefore the best rate obtained by this method is 
$$\overline{\alpha} = \frac{\log{\frac{\mu}{\lambda_0}}}{\log \sqrt{\frac{\mu}{\lambda}}} \left( \sqrt{\mu} - \sqrt{\lambda} \right)^2.$$
\end{proof}

\begin{remark}
The function
$$g(\lambda_0) = \begin{cases}
\left( \sqrt{\mu} - \sqrt{\lambda} \right)^2 & \text{ if } \lambda_0 \leq \sqrt{\lambda \mu}\\
\frac{\log{\frac{\mu}{\lambda_0}}}{\log \sqrt{\frac{\mu}{\lambda}}} \left( \sqrt{\mu} - \sqrt{\lambda} \right)^2 & \text{ if } \lambda_0 > \sqrt{\lambda \mu}
\end{cases}
$$
is continuous. In other words, the convergence rate changes continuously in $\lambda_0$.
\end{remark}

For certain values of $(\lambda_0, \lambda, \mu)$ this rate is better than the rate previously computed, $\alpha = \lambda + \mu - \lambda_0 - \frac{\lambda \mu}{\lambda_0}$. However, $\overline{\alpha} < \alpha^{\star}$ when $\lambda_0 > \sqrt{\lambda \mu}$, so there is still a gap, and we do not know the best convergence rate in this case. We suspect that the rate $\overline{\alpha}$ not the best possible, since the step from expression \eqref{loose1} to expression \eqref{loose2}, which exchanges the order of a supremum with a sum, can be quite loose.

\subsection*{Acknowledgements}
This work was partially supported by the Singapore National Research Foundation through the Singapore-MIT Alliance for Research and Technology (SMART) Centre for Future Urban Mobility (FM). Julia Gaudio is supported by a Microsoft Research PhD fellowship.
\small
\bibliographystyle{plain}
\bibliography{MIT_Project_1}

\end{document}